\documentclass{elsarticle}
\usepackage{amssymb, amsmath, latexsym, enumerate}
\usepackage{graphicx}

\numberwithin{equation}{section}

\newtheorem{theorem}{Theorem}
\newtheorem{lemma}[theorem]{Lemma}
\newdefinition{remark}[theorem]{Remark}
\newproof{proof}{Proof}
\newproof{pot}{Proof of Theorem \ref{main}}

\newcommand{\Z}{\mathbb{Z}}
\newcommand{\Q}{\mathbb{Q}}

\newcommand{\A}{\mathbb{A}}
\newcommand{\N}{\mathfrak{N}}
\newcommand{\n}{\mathfrak{n}}
\newcommand{\s}{\mathfrak{s}}
\newcommand{\gen}{\text{gen}}
\newcommand{\spn}{\text{spn}}

\newcommand{\lan}{\langle}
\newcommand{\ran}{\rangle}

\DeclareMathOperator{\ord}{ord}
\DeclareMathOperator{\rad}{sf}

\begin{document}

\begin{frontmatter}

\title{A characterization of almost universal ternary inhomogeneous quadratic polynomials with conductor 2}
\author{Anna Haensch \corref{cor1}}

\cortext[cor1]{Corresponding Author: Anna Haensch, Duquesne University, Department of Mathematics and Computer Science, 600 Forbes Ave., Pittsburgh, PA, 15282, USA; Email, haenscha@duq.edu; Phone, +001 412 396 4851}


\begin{abstract}
An integral quadratic polynomial (with positive definite quadratic part) is called almost universal if it represents all but finitely many positive integers.  In this paper, we provide a characterization of almost universal ternary quadratic polynomials with conductor 2.  
\end{abstract}

\begin{keyword}
primitive spinor exceptions \sep ternary quadratic forms 

\MSC[2010] 11E12 \sep11E20 \sep 11E25
\end{keyword}

\end{frontmatter}

\section{Introduction}
For a polynomial $f(x_1,...,x_n)$ with rational coefficients and an integer $a$, we say that $f$ represents $a$ if the diophantine equation $f(x_1,...,x_n)=a$ has a solution over the integers.  One particularly interesting question asks, when is a polynomial {\em almost universal}; that is, when does a polynomial represent all but finitely many natural numbers?   In the specie when $f$ is a quadratic forms, this question has attracted a great deal of interest over the last years. 

Given a quadratic map $Q$ and $\Z^n$ in the standard basis, the pair $(\Z^n,Q)$ is a $\Z$-lattice of rank $n$, which we denote by $N$.  Since the representation behavior for indefinite lattices is well understood, all lattices in this paper are assumed to be positive definite;  for the indefinite case the reader is referred to the survey paper by Hsia \cite{H99}.

A homogeneous integral quadratic polynomial can always be viewed as a quadratic lattice.  For rank greater than 4, Tartakowsky's results in \cite{T29} imply that a lattice is almost universal if it is universal over $\Z_p$ for every prime $p$.  In the quaternary case, Bochnak and Oh \cite{BO09} give an effective method to determine when a lattice is almost universal, resolving an investigation first initiated by Ramanujan in \cite{R00}.  For the ternary case, it is a well known consequence of Hilbert Reciprocity that a positive definite ternary $\Z$-lattice is anisotropic at an odd number of finite primes, and therefore is is not universal at these primes.  Hence the lattice fails to represent an entire square class in $\Q_p/\Q_p^\times$ at these primes, and hence cannot be almost universal. 

Therefore, we turn out attention to inhomogeneous quadratic polynomials of the form
\[
f(x)=Q(x)+L(x)+c,
\]
where $Q(x)$ is a quadratic form, $L$ is a linear form, and $c$ is a constant.  It is not a surprise that we can study the arithmetic of these polynomials from the geometric perspective of quadratic spaces and lattices.  Indeed, $Q$ can be viewed as the quadratic map on $N=(\Z^n,Q)$, and associated to $Q$ is the symmetric bilinear map $B$.  Under the assumption that $Q$ is positive definite, $L(x)=2B(\nu,x)$ for a unique choice of vector $\nu$ in $\mathbb Q N$ which is the quadratic space underlying $N$.  The choices for $\nu$ and $N$ are completely determined by the coefficients of $Q$ and $L$.  Since the constant $c$ does not contribute anything essential to the arithmetic of $f$, there is no harm in assuming that it is equal to zero.   Thus,  an integer $a$ is represented by $f(x)$ if and only if $Q(\nu)+a$ is represented by the coset $\nu+N$.  In general, there is no local-global principle for representations of integers by cosets of quadratic lattices.  However, when $n\geq 4$, Chan and Oh \cite[Theorem 4.9]{CO12} show how the asymptotic local-global principles for representations with approximation property by J\"ochner- Kitaoka \cite{JK94} and by Hsia-J\"ochner \cite{HJ97} lead to an asymptotic local-global principle for representations of integers by cosets.  Therefore we restrict the discussions of this paper to the ternary case.

Given a lattice $N$ and a vector $\nu\in \Q N$, we define the conductor $\mathfrak{m}$ of $\nu+N$ as in \cite{H13}; that is, the minimal integer for which $\mathfrak{m}\nu\in N$.  In \cite{H13} we give a characterization of almost universal ternary inhomogeneous quadratic polynomials where $\mathfrak{m}$ is an odd prime power.  In this paper, we will restrict our attention to $\mathfrak{m}=2$. Therefore, we will impose the following assumption through this paper,
\[
2\nu\in N, \tag{I}\label{R1}
\]
implying that $[M:N]=2$, where $M:=\Z\nu+N$. The main difference between the odd case and the even case is that $N_2$ is not necessarily diagonalizable.  Consequently, we will impose a few additional assumptions for our convenience; letting $Q(\nu+N)$ and $B(\nu,N)$, respectively, denote the $\Z$-ideals generated by $Q(\nu+x)$ and $B(\nu,x)$ for all $x\in N$, we require that    
\[
Q(\nu+N)\subseteq \Z \text{ and }B(\nu,N)\subseteq \Z.\tag{II}\label{R2}
\]
An immediate consequence of \eqref{R2} is that $Q(\nu)\in \Z$, and so we define a positive integer $\beta:=\ord_2(Q(\nu))$ and let $2^\beta\epsilon:=Q(\nu)$ where $\epsilon\in \Z_2^\times$.  It also follows from \eqref{R2} that $\n(\nu,N)\subseteq \Z$, where $\n(\nu,N)$ denotes the integral ideal generated by $Q(x)+2B(\nu,x)$ for all $x\in N$.  If $Q(x)+2B(\nu,x)$ were almost universal, then the ternary $\Z$-lattice $M$ would be almost universal, which cannot happen.  Therefore it is logical to impose one final restriction,  
\[
\n(\nu,N)=2^\alpha \Z\text{ with }\alpha>0,\tag{III}\label{R3}
\]
and now replace $f(x)$ with $H(x):=\frac{Q(x)+2B(\nu,x)}{2^\alpha}$.  For a classical example of this type of polynomial, we turn to the sum of three triangular numbers, which can be written
\[
\frac{x(x+1)}{2}+\frac{y(y+1)}{2}+\frac{z(z+1)}{2}=\frac{4(x^2+y^2+z^2)+4(x+y+z)}{8},
\]
giving $\alpha=3$ and $N\cong\lan4,4,4\ran$ in a basis $\{e_1,e_2,e_3\}$ and $\nu=\frac{e_1+e_2+e_3}{2}$.  We note here that there are of course many possible choices for $\nu$, $N$, and consequently $\alpha$, but this particular choice satisfies \eqref{R1}, \eqref{R2}, and \eqref{R3}. 

For the remainder of the document we assume that \eqref{R1}, \eqref{R2}, and \eqref{R3} hold.  We will define an integer $\lambda$ by 
\[
\lambda:=\begin{cases}
1 & \text{ if }\ord_2(dN)-3\beta\text{ is even}\\
2 & \text{ if }\ord_2(dN)-3\beta\text{ is odd}.
\end{cases}
\] 
Let $\rad(dN)'$ denote the odd square-free part of $dN$. Now we state the main theorem, the proof of which will be given in section 3, after establishing several technical lemmas.  Below, the scale of $N$, denoted $\s(N)$, is the $\Z$-ideal generated by $B(x,y)$ for all $x,y\in N$, and the norm of $N$, denoted $\n(N)$, is the $\Z$-ideal generated by $Q(x)$ for all $x\in N$, following the notional set forth in O'meara's infuential text on quadratic forms \cite{OM}. 
 
\begin{theorem}\label{main}
$H(x)$ is almost universal if and only if $N_p$ represents all of $\Z_p$ for every odd prime $p$, and one of the following holds:
\begin{enumerate}[(1)]
\item $\alpha=\beta+1$, and

\begin{enumerate}[(a)]
\item $B(\nu,N_2)=2^{\beta-1}\Z_2$; or,
\item $B(\nu,N_2)\subseteq 2^{\beta}\Z_2$, and 
\begin{enumerate}[(i)]
\item $2\s(N)=\n(N)=2^{\beta+2}\Z$; or, 
\item $N_2$ diagonalizable and $\ord_2(dN)=3+3\beta$; or,
\item $N_2$ is diagonalizable, $\ord_2(dN)=5+3\beta$. and $B(\nu,N_2)=2^{\beta+1}\Z$. 
\end{enumerate}
\end{enumerate}

\item $\alpha=\beta+2$, and

\begin{enumerate}[(a)]
\item $B(\nu,N_2)=2^{\beta}\Z_2$, and
\begin{enumerate}[(i)]
\item $\ord_2(dN)-3\beta$ is odd; or 
\item $\ord_2(dN)-3\beta=4$; or,
\item $\rad(dN)'$ is divisible by a prime $p$ for which $\left(\frac{-\lambda}{p}\right)=-1$; or,
\item $N_2$ has a binary Jordan component with the square free part of its discriminant congruent to $5\mod 8$; or,
\end{enumerate}
\item $B(\nu,N_2)=2^{\beta+1}\Z_2$, $\n(G_2)=2^{\beta+2}\Z_2$, where $G_2$ is the orthogonal complement of $\nu$ in $N_2$, and 
\begin{enumerate}[(i)]
\item $\ord_2(dN)-3\beta$ is odd; or,
\item $\ord_2(dN)-3\beta=6$; or,
\item $\rad(dN)'$ is divisible by a prime $p$ for which $\left(\frac{-\lambda}{p}\right)=-1$. 
\end{enumerate}
\end{enumerate}

\item $\alpha=\beta+3$, and
 
 \begin{enumerate}[(a)]
\item $G_2$ is not diagonalizable; or,
\item $\n(G_2)=2^{\alpha}\Z_2$, and $\ord_2(dN)-3\beta$ is even, or $\ord_2(dN)=9+3\beta$; or,
\item $\n(G_2)=2^{\alpha+1}\Z_2$ and $\ord_2(dN)-3\beta$ is odd; or,
\item $\rad(dN)'$ is divisible by a prime $p$ satisfying $\left(\frac{-\lambda}{p}\right)=-1$; or,
\item $\rad(dN)'\not\equiv Q(\nu)'\mod 8$; or,
\item $\n(G_2)=2^{\alpha}\Z_2$ and $2^\alpha Q(\nu)$ is not represented by $G_2$. 
\end{enumerate}

\item $\alpha=\beta+2$ or $\beta+3$, and $\frac{2^\beta\rad(dN)'-Q(\nu)}{2^{\alpha}}$ is represented by $H(x)$. 
\end{enumerate}
\end{theorem}

\section{Preliminaries}

Henceforth, the language of quadratic spaces and lattices as in \cite{OM} will be adopted, and the notation will follow that used in \cite{H13}. Any unexplained notation and terminology can be found there and in \cite{OM}.   

The subsequent discussion involves the computation of the spinor norm groups of local integral rotations and the relative spinor norm groups of primitive representations of integers by ternary quadratic forms.   The formulae for all these computations can be found in \cite{EH75}, \cite{EH78}, \cite{EH94}, and \cite{H75}.  A correction of some of these formulae can be found in \cite[Footnote 1]{CO09}.  Following the notation set forth in \cite[\S 55]{OM}, the symbol $\theta$ always denotes the spinor norm map.  If $t$ is an integer represented primitively by $\gen(K)$ and $p$ is a prime, then $\theta^*(K_p, t)$ is the primitive relative spinor norm group of the $\Z_p$-lattice $K_p$. If $E$ is a quadratic extension of $\mathbb Q$, $\N_p(E)$ denotes the group of local norms from $E_{\mathfrak p}$ to $\mathbb Q_p$, where $\mathfrak p$ is an extension of $p$ to $E$.  

\begin{lemma}\label{L0}
 If $H(x)$ is almost universal, then $N_p$ represents all of $\Z_p$ whenever $p$ is odd, and consequently, 
\begin{enumerate}[(1)]
\item$M_p\cong \lan1,-1,-dM \ran$ and $\theta(O^+(M_p))\supseteq \Z_p^\times$ for all odd primes $p$; 
\item$\Q_2M=\Q_2N$ is anisotropic; and, 
\item if $t$ is a primitive spinor exception of $\gen(M)$, then $E:=Q(\sqrt{-tdM})$ is either $Q(\sqrt{-1})$ or $\Q(\sqrt{-2})$.
\end{enumerate}
\end{lemma}

\begin{proof}
For (1) and (2), see \cite[Lemma 2.1]{CO09}; for (3) see \cite[Lemma 2.3]{CO09}.
\end{proof}

As an immediate consequence of Lemma \ref{L0}, if $H(x)$ is almost universal, then $\beta<\alpha$, or else $\Q_2N$ is isotropic.  Furthermore, $Q(2\nu)+2B(\nu,2\nu)\in 2^{\beta+3}\Z$, and consequently,  $\beta<\alpha\leq \beta+3$. 

If $H(x)$ is almost universal, then by Lemma \ref{L0}, $\n(N_p)=\Z_p$ for all odd $p$.  Consequently, if $H(x)$ is almost universal, then $\n(N)\subseteq 2\Z$, since $\n(\nu,N)=2^\alpha\Z$ for $\alpha>0$.  Suppose that $\s(N_2)=\Z_2$.  Then, since $N_2$ is anisotropic by Lemma \ref{L0}, and since $\n(N_2)\not\subseteq \s(N_2)$ the bilinear form on $N_2$ must represent a unit in $\Z_2^\times$, hence it represents an element in $\Z_2$ which is not represented by the quadratic form on $N_2$.  Consequently, $N_2\cong\mathbb{A}\perp4\lan\eta\ran$ in the basis $\{e_1,e_2,e_3\}$, where $\eta\in \Z_2^\times$ and $\mathbb{A}:=A(2,2)$ (cf. \cite[93:11]{OM}).  But then, $M_2\cong\mathbb{A}\perp\lan\epsilon\ran$ in the basis $\{e_1,e_2,\nu\}$ fails to represent all $\epsilon+4\delta$ with $\delta\in \Z_2^\times$.  Therefore, we may assume that $\s(N_2)\subseteq 2\Z_2$.

\begin{lemma}\label{F0}
If $N_p$ represents all of $\Z_p$ whenever $p$ is odd, then
\[ 2^{\beta+1}\Z_2\subseteq B(\nu,N_2)\subseteq 2^{\beta-1}\Z_2.
\]
\end{lemma}

\begin{proof}
Suppose that $N_p$ represents all of $\Z_p$ whenever $p$ is odd, and therefore $N_2$ is anisotropic by Lemma \ref{L0}.  The left-hand containment is obvious since $B(\nu,N_2)$ contains $2Q(\nu)\in 2^{\beta+1}\Z_2^\times$. For the sake of contradiction, suppose that the right-hand containment does not hold, that is, suppose that $2^{\beta-2}\Z_2\subseteq B(\nu,N_2)$. Since $\alpha>\beta$ by Lemma \ref{L0}, this implies that $\n(N_2)=2^i\Z_2$, where $i<\beta$. 

First suppose that $\s(N_2)=\n(N_2)$.  Then $N_2\cong\lan2^i\eta\ran\perp K_2$ in a basis $\{e_1,e_2,e_3\}$, where $\eta\in \Z_2^\times$, and $K_2$ is some binary lattice. Then $\nu=\frac{ae_2+be_2+ce_3}{2}$, where $a,b,c\in \Z$, and $Q(e_1)+2B(\nu,e_1)=2^i\eta(1+a^2)$. But since $i<\beta<\alpha$, this value is never in $2^\alpha \Z$, regardless of the parity of $a$, contradicting assumption (\ref{R3}). 

Next, we suppose that $2\s(N_2)=\n(N_2)$, then $N_2\cong2^{i-1}\A\perp\lan2^j\eta\ran$ in a basis $\{e_1,e_2,e_3\}$, where $i-1< j$.  Since $N_2$ is anisotropic we may conclude that $i-1\equiv j\mod 2$, so in particular $i+1\leq j$.  Letting $\nu=\frac{ae_1+be_2+ce_3}{2}$ with $a,b,c\in \Z$, we get 
\[
Q(\nu)=2^{i-2}(a^2+b^2+ab)+2^{j-2}\eta c^2\in 2^\beta\Z_2^\times.
\]
Since $i-2<\beta$ and since $i-2< j-2$, we may conclude that both $a$ and $b$ are even, and since $j-2\neq i$ we may go further to say that at least one of $a$ or $b$ is congruent to $0\mod 4$. Without loss of generality, we will suppose that $b\equiv 0\mod 4$, then $Q(e_1)+2B(\nu,e_1)=2^{i-1}[2(1+a)+b]\in 2^{i}\Z_2^\times$, contradicting (\ref{R3}). 
\end{proof}

\begin{lemma}\label{L3}
Let $\omega=\nu+x_0$, where $x_0\in N$, and define $H'(x)=\frac{1}{2^\alpha}[Q(x)+2B(\omega,x)]$.  Then, 
\begin{enumerate}[(1)]
\item if $\n(\nu,N)=2^\alpha \Z$, then $\n(\omega, N)=2^\alpha \Z$; and,
\item if $H(x)$ is almost universal, then $H'(x)$ is almost universal.  
\end{enumerate}
\end{lemma}

\begin{proof}
See \cite[Lemma 3]{H13} 
\end{proof}

\begin{remark}\label{CL3}
The rest of this paper is to demonstrate the almost universality of $H(x)$ under some arithmetic conditions imposed on $\nu$ and $N$.  In view of Lemma \ref{L3}, we can always change $\nu$ to $\nu+x_0$ with $x_0\in N$, as long as the arithmetic conditions are unchanged. 
\end{remark}

In view of Lemma \ref{L3} and Remark \ref{CL3}, once we've fixed a basis $\{e_1,e_2,e_3\}$ for $N_2$, it is always possible to write $\nu=\frac{ae_1+be_2+ce_3}{2}$ where $0\leq a,b,c\leq 1$.  Given our assumptions \eqref{R1}, \eqref{R2}, \eqref{R3} and Lemma \ref{L0}, we now establish some arithmetic conditions on $\n(N)$, $\s(N)$ and $B(\nu,N_2)$.  In what follows, $\eta, \gamma,\mu$ always denote units in $\Z_2^\times$. 

\begin{lemma}\label{scale}
If $H(x)$ is almost universal and $\n(N_2)=2^{\beta+1}\Z_2$, then $\s(N_2)=\n(N_2)$. 
\end{lemma}

\begin{proof}
Suppose that $\n(N_2)=2^{\beta+1}\Z_2$, and for the sake of contradiction, we will suppose that $\s(N_2)=2^\beta\Z_2$.  Then $N_2\cong2^\beta\A\perp\lan2^i\eta\ran$ in a basis $\{e_1,e_2,e_3\}$, where $\beta<i$. But since $N_2$ must represent $Q(2\nu)\in 2^{\beta+2}\Z_2^\times$, it follows that $i=\beta+2$.  But now $M_2\cong2^\beta \A\perp \lan2^\beta\epsilon\ran$ in a basis $\{e_1,e_2,\nu\}$.  But $M_2$ fails to represent all $2^\beta\epsilon+2^{\beta+j}\gamma$ for every $\gamma\in\Z_2^\times$, and $j$ even. Therefore $H(x)$ is not almost universal. 
\end{proof}

\begin{lemma}
If $H(x)$ is almost universal, then $B(\nu,N_2)=B(\omega,N_2)$ for any $\omega\in \nu+N$.  
\end{lemma}

\begin{proof}
Suppose that $H(x)$ is almost universal, and let $\omega\in \nu+N$.  Combining (\ref{R3}) and Lemma \ref{L3} we know that $\n(\nu,N)=\n(\omega,N)=2^\alpha\Z$. From Lemma \ref{F0} we know that $B(\nu,N_2)=2^{\beta+i}\Z_2$, where $i=-1,0,1$.  If $\alpha\neq\beta+1$, then $\beta+i+1\leq\alpha$, and thus $2B(\nu,N_2)=2^{\beta+i+1}\Z_2$ implies that $\n(N_2)=2^{\beta+i+1}\Z_2$. Therefore $\n(\omega,N_2)=2^\alpha\Z_2$ implies that $B(\omega,N)=2^{\beta+i}\Z_2$, and hence $B(\nu,N_2)=B(\omega,N_2)$. 

When $\alpha=\beta+1$ the argument follows as above when $i=-1$.  If $i=0$, then $B(\nu,N_2)=2^{\beta}\Z_2$ implies that $\n(N_2)=2^{\beta+1}\Z_2$ or $2^{\beta+2}\Z_2$.  If $\n(N_2)=2^{\beta+1}\Z_2$, then it follows from Lemma \ref{scale} that $\s(N_2)=2^{\beta+1}\Z_2$.  In this case, $B(\omega,x)=B(\nu,x)+B(x_0,x)$ for $x_0\in N$ implies that $B(\omega,N_2)=2^\beta\Z_2$.  On the other hand, if $\n(N)=2^{\beta+2}\Z_2$, then $\n(\omega,N_2)=2^{\beta+1}\Z_2$ implies that $B(\omega,N_2)=2^\beta\Z_2$. 

If $i=1$, then $\n(\nu,N_2)=2^{\beta+1}\Z_2$ implies that $\n(N_2)=2^{\beta+1}\Z_2$. Therefore $B(\omega,N_2)\subseteq 2^\beta\Z_2$. But if $B(\omega,N_2)=2^\beta\Z_2$, then this would imply that $B(\nu,x)+B(x_0,x)\in 2^\beta\Z$, which is impossible in view of Lemma \ref{scale}.  Thus $B(\omega,N_2)=2^{\beta+1}\Z_2$. 
\end{proof}

\begin{lemma}\label{L1.1}
If $N_p$ represents all of $\Z_p$ for every odd prime $p$, then the following hold:
\begin{enumerate}[(1)]
\item If $\alpha=\beta+3$, then $B(\nu,N_2)=2^{\beta+1}\Z$; 
\item If $\alpha=\beta+2$, then $B(\nu,N_2)=2^{\beta}\Z$ or $2^{\beta+1}\Z$;  
\end{enumerate}
and in both cases $\n(N)=\s(N)$. 
\end{lemma}

\begin{proof}
We will suppose throughout that $N_p$ represents all of $\Z_p$ for every $p$ odd.  From Lemma \ref{L0}, this implies that $N_2$ is anisotropic.  Suppose that $\alpha=\beta+2$ or $\beta+3$, and $B(\nu,N_2)=2^{\beta+i}\Z_2$, where $i=-1,0,1$ by Lemma \ref{F0}.   Under assumption (\ref{R3}), we get that $\n(N)=2^{\beta+i+1}\Z$. 

Suppose that $\n(N)=\s(N)$.  Then, $N_2\cong\lan 2^{\beta+i+1}\eta\ran\perp K_2$ in a basis $\{e_1,e_2,e_3\}$, where $K_2$ is some binary lattice.  Letting $2\nu=ae_1+be_2+ce_3$, where $0\leq a,b,c\leq 1$ by Remark \ref{CL3}, we get $Q(e_1)+2B(\nu,e_1)=2^{\beta+i+1}\eta(1+a)$.  But since $a$ is either 0 or 1, we have that $\ord_2(2^{\beta+i+1}\eta(1+a))\leq \beta+i+2$.  Therefore, since (\ref{R3}) must hold, we conclude that when $\alpha=\beta+2$ then $i=0$ or $1$, and when $\alpha=\beta+3$ then $i=1$.  

Next, we will deal with the case where $\n(N)=2\s(N)$.  In this case, $N_2\cong2^{\beta+i}\A\perp \lan2^{j}\eta\ran$ in a basis $\{e_1,e_2,e_3\}$, where $\beta+i+2\leq j$ since $N_2$ is anisotropic.  Setting $2\nu=ae_1+be_2+ce_3$ with $0\leq a,b,c\leq 1$, we get that $Q(\nu)=2^{\beta+i-1}\left[(a^2+b^2+ab)+2^{j-(\beta+i+1)}\eta c^2\right]$.  Since $Q(\nu)\in 2^\beta\Z_2^\times$, this implies that when $i=-1,0$ then $a=b=0$, and when $i=1$, then at least one of $a$ or $b$ must be odd, without loss of generality we will suppose that $b=1$. But now $Q(e_1)+2B(\nu,e_1)=2^{\beta+i}(2+2a+b)$, and for any choice of $i$ we have $\ord_2(2^{\beta+i}(2+2a+b))\leq \beta+1$.  Therefore, by assumption (\ref{R3}), we conclude that $\alpha$ cannot be $\beta+2$ or $\beta+3$. 

\end{proof}

\begin{remark}\label{RL1.1.1}
When $N_p$ represents all of $\Z_p$ for every odd prime $p$, we notice that when $\alpha=\beta+2$ and $B(\nu,N_2)=2^\beta\Z$, then $N_2$ must be diagonalizable, or else $N_2$ cannot represent $Q(2\nu)\in 2^{\beta+2}\Z_2^\times$. 
\end{remark}

\begin{remark}\label{RL1.1.2}
When $\alpha=\beta+1$ then if $N_2$ is not diagonalizable, it can easily be shown that for any choice of basis $\{e_1,e_2,e_3\}$ for $N_2$, we get $B(\nu,e_n)\in 2^{\beta+1}\Z$ for $n=1,2,3$. Therefore, when $\alpha=\beta+1$ and $B(\nu,N_2)=2^{\beta-1}\Z_2$, then $N_2$ is diagonalizable. 
\end{remark}

\begin{lemma}\label{L6}
Suppose that $Q(\nu)\in \Z_2^\times$.  If $\alpha=1$ and $N_p$ represents every element in $\Z_p$ for every odd prime $p$, then $\gen(M)$ primitively represents every unit in $\Z_2^\times$ if one of the following holds:
\begin{enumerate}[(1)]
\item $2\s(N)=\n(N)=4\Z$; or,
\item $ord_2(dN)=3$; or,
\item $ord_2(dN)=5$ and $B(\nu,N_2)\neq\Z_2$. 
\end{enumerate}
Furthermore, if (1), (2), and (3) all fail, then $H(x)$ is not almost universal. 
\end{lemma}

\begin{proof}
Suppose that $N_p$ represents every $p$-adic integer for every odd prime $p$, then  $N_p\cong\lan1,-1,-dN\ran$ by Lemma \ref{L0}, and so every unit in $\Z_2^\times$ is represented primitively by $N_p$ at every odd prime $p$.  

If (1) holds, then we may assume that $N_2\cong 2\mathbb{A}\perp 2^i\lan\eta\ran$ in a basis $\{e_1,e_2,e_3\}$ where $i>1$.  Assuming that $\nu=\frac{ae_1+be_2+ce_3}{2}$ with $0\leq a,b,c\leq 1$, since $\epsilon\in \Z_2^\times$, we may conclude that at least one of $a,b$ is odd, without loss of generality we will assume that $a=1$.  But now $M_2\cong\Z_2[\nu,e_2,e_3]$, which contains the binary sublattice $\Z_2[\nu,e_2]$.  But $B(\nu,e_2)=1+2b\in \Z_2^\times$, and therefore $\Z_2[\nu,e_2]\cong\lan\epsilon, \epsilon (-1+4\epsilon)\ran$, which primitively represents all units in $\Z_2^\times$. 

If part (2) holds, then $N_2\cong\lan2\eta,2\gamma,2\mu\ran$ in a basis $\{e_1,e_2,e_3\}$.  Let $\nu=\frac{ae_1+be_2+ce_3}{2}$ where $0\leq a,b,c\leq 1$.  Then, since $\epsilon\in\Z_2^\times$, we may assume without loss of generality that $a=b=1$ and $c=0$.   But now $M_2\cong \Z_2[\nu,e_2,e_3]\cong\lan \epsilon, \epsilon(-1+2\epsilon\mu), 2\mu\ran$, which primitively represents all units in $\Z_2^\times$.  

If part (3) holds, then $\n(N)=2\Z$.  If $N_2$ is not diagonalizable, then $N_2\cong\lan2\eta\ran\perp 4\A$ or $\lan2^5\eta\ran\perp \A$, both of which are isotropic.  Therefore, we assume $N_2\cong\lan 2\eta,2^i\gamma,2^j\mu\ran$ in a basis $\{e_1,e_2,e_3\}$ where $i\leq j$.  We let $\nu=\frac{ae_1+be_2+ce_3}{2}$ where $0\leq a,b,c\leq 1$, and since $B(\nu,N_2)=2\Z_2$, it is immediate that $a=0$. Therefore, since $\epsilon\in \Z_2^\times$ and $\ord_2(dN)=5$, it follows that $i=j=2$ and exactly one of $b,c$ is odd, so without loss of generality we will suppose that $b=1$ and $c=0$.  Then, $M_2\cong\Z_2[\nu,e_1,e_3]$ which is isometric to $\lan\epsilon, 2\eta, 4\mu\ran$, which primitively represents all units in $\Z_2^\times$. 

Now we will assume that parts (1), (2) and (3) all fail, and we will show that $H(x)$ is not almost universal.  Since $\s(N)\subseteq 2\Z$ and since $N$ must represent $4\epsilon$, if $2\s(N)=\n(N)$ then the only choice is for $\n(N)=4\Z$.  If (1), (2), and (3) all fail, then either $N_2$ is diagonalizable and then $\ord_2(dN)>5$, or $\ord_2(dN)=5$ and $B(\nu,N_2)=\Z_2$; or, $N_2$ is not diagonalizable, and then $\s(N_2)=\n(N_2)=2\Z$ or $4\Z$. 

First we will suppose that $\ord_2(dN)=5$, $B(\nu,N_2)=\Z$, and $N_2$ is diagonalizable.  Under these assumptions, the only possibility is that $N_2\cong\lan2\eta, 2\gamma,8\mu\ran$, or else $B(\nu,N_2)=2\Z_2$.  Then we may assume that $\nu=\frac{e_1+e_2+ce_3}{2}$ where $0\leq c\leq 1$, and from here it is not difficult to see that $H(x)$ only represents even integers, and hence is not almost universal. 

If $N_2$ is diagonalizable, with $\ord_2(dN)>5$ and $N_2$ is of the form $\lan2\eta, 2\gamma,2^j\mu\ran$ with $j\geq 4$, then the argument follows as in the previous paragraph.  Therefore, we may assume that $N_2\cong \lan2\eta, 2^i\gamma,2^j\mu\ran$ in a basis $\{e_1,e_2,e_3\}$, where $3\leq i\leq j$.  Since $\epsilon\in \Z_2^\times$, it follows that $\nu=\frac{be_2+ce_3}{2}$ where $0\leq b,c\leq 1$, and thus $\{\nu,e_2,e_1\}$ is a basis for $M_2$.  Therefore, $M_2\cong\lan\epsilon,\epsilon(2^i\gamma\epsilon-(2^{i-1}\gamma)^2),2\eta\ran$, which clearly only represents units in the square classes of $\epsilon$ and $\epsilon+2\eta$.  Therefore, $H(x)$ cannot be almost universal. 

Now suppose that $N_2$ is not diagonalizable.  Then from the failure of (1), (2), and (3), we may conclude that $N_2\cong\lan2^i\eta\ran\perp 2^j\mathbb{A}$, where $1\leq i\leq j$ and $i$ and $j$ have the same parity by Lemma \ref{L0}.  If $\n(N)=2\Z$, then $i=1$, in which case $j\geq 3$ is odd and $N_2$ fails to represent $\epsilon$. Therefore we will suppose that $\n(N_2)=4\Z$.  Then, $2B(\nu,N_2)\subseteq B(2\nu,N_2)\subseteq 4\Z$ since $i\leq j$, and therefore $H(x)$ only represents even integers.  
\end{proof}

\begin{lemma}\label{L7}
Suppose that $Q(\nu)\in \Z_2^\times$.  If $\alpha=2,3$ and $N_p$ represents every element in $\Z_p$, then $\gen(M)$ primitively represents $\epsilon+2^\alpha n$ for all positive $n\in \Z$ if and only if one of the following holds: 
\begin{enumerate}[(1)]
\item $\alpha=2$, and $B(\nu,N_2)=\Z_2$; or,
\item $\alpha=2$, $B(\nu,N_2)\neq\Z_2$ and $\n(G_2)=4\Z_2$, where $G$ is the orthogonal complement of $\nu$ in $N$; or,  
\item $\alpha=3$ and $B(\nu,N_2)\neq\Z_2$.
\end{enumerate}
\end{lemma}

\begin{proof}
Suppose that $N_p$ represents all of $\Z_p$ for every odd prime $p$.  Then, for every odd prime $p$, $N_p\cong\lan1,-1,-dN\ran$ by Lemma \ref{L0}, and so $\epsilon+2^\alpha n$ is represented primitively by $N_p$ at every odd prime $p$.  

Suppose that $\alpha=2$ and $B(\nu,N_2)=\Z_2$.  Then $\n(N)=2\Z$, and by Remark \ref{RL1.1.1} we may assume that $N_2\cong\lan2\eta,2^i\gamma,2^j\mu\ran$ in a basis $\{e_1,e_2,e_3\}$, where $1\leq i\leq j$.  By Remark \ref{CL3} we may let $\nu=\frac{ae_1+be_2+ce_3}{2}$, with $0\leq a,b,c\leq 1$.   Since $\epsilon\in \Z_2^\times$, either $a=1$, in which case $b=1$ and $i=1$, or $a=0$, in which case $b=c=1$ and $i=j=1$.  But in either case, $M_2$ contains the binary sublattice $\Z_2[\nu,e_2]$ which clearly represents all units congruent to $\epsilon\mod 4$.  Therefore, $M_2$ primitively represents all units of the form $\epsilon+4n$; thus $\gen(M)$ primitively represents $\epsilon+4n$ for all $n\geq 1$.   

For parts (2) and (3) we observe that when $B(\nu,N_2)=2\Z_2$, then whether $\alpha=2$ or $3$, it must be the case that $\n(N)=4\Z$.  So in either case $\Z[2\nu]\cong\lan4\epsilon\ran$ splits $N_2$ as an orthogonal summand, and therefore,  $M_2\cong\lan\epsilon\ran\perp G_2$.  In either case it is a consequence of the Local Square Theorem \cite[63:1]{OM} that $\epsilon+2^\alpha n$ is represented primitively by $M_2$. Therefore, $\epsilon+2^\alpha n$ is represented primitively by $\gen(M)$ for every $n\geq 1$. 

Suppose that parts (1)-(3) all fail.  If $\alpha=2$, then it follows from that failure of (1) and (2) that $B(\nu,N_2)=2\Z_2$ and $\n(G_2)=8\Z_2$. However, in this case it is clear that $M_2$ will only represent units in the square class of $\epsilon$.  If $\alpha=3$, then from the failure of (3) it follows that $B(\nu,N_2)=\Z_2$.  However, this case cannot occur by Lemma \ref{L1.1}.  

\end{proof}

For the following lemma, we define an integer $\delta_N$ by 
\[
\delta_N:=\begin{cases}
1 & \text{ if }\ord_2(dN)\text{ is even,}\\
2 & \text{ if }\ord_2(dN)\text{ is odd.}
\end{cases}
\]
If $N_p$ represents all of $\Z_p$ for every odd prime $p$, then for any odd primitive spinor exception $t$, $E=\Q(\sqrt{-\delta_N})$, were $E$ is as defined in Lemma \ref{L0}.  

\begin{lemma}\label{L9}
Suppose that $Q(\nu)\in \Z_2^\times$.  Suppose that $N_p$ represents all of $\Z_p$ for every odd prime $p$.  If $\rad(dN)$ is divisible by an odd prime $p$ satisfying $\left(\frac{-\delta_N}{p}\right)=-1$, then $\gen(M)$ has no odd primitive spinor exceptions. 
\end{lemma}

\begin{proof}
Suppose that $\gen(M)$ has an odd primitive spinor exception, and hence $E=\Q(\sqrt{-\delta_N})$.  For any odd prime $p$ satisfying $\left(\frac{-\delta_N}{p}\right)=-1$, $E_p/\Q_p$ is an unramified quadratic extension and hence we are in the setting of \cite[Theorem 1]{EH94}. Since for any such $p$ we have $\theta(O^+(M_p))\subseteq \mathfrak{N}_p(E)$, it must follow that $\ord_p(dN)$ is even, and consequently $p\nmid \rad(dN)$.   Therefore the contrapositive must hold, namely, if $\rad(dN)$ is divisible by a prime $p$ satisfying $\left(\frac{-\delta_N}{p}\right)=-1$, then $\gen(M)$ cannot have any odd primitive spinor exceptions. 
\end{proof}

If we have a lattice $K\cong\lan1,2^i\gamma,2^j\mu\ran$ with $0<i< j$ and $\gamma,\mu\in \Z_2^\times$, and $K$ is not of {\em type $E$} as defined in \cite{EH78}, we will compute the spinor norm for $K$ as follows.  Defining sublattices $U$ and $W$ as  
\[
U\cong\lan1,2^i\gamma\ran\text{ and }W\cong2^i\gamma\lan1,2^{j-i}\gamma\mu\ran,
\]
then by \cite[Theorem 2.7]{EH75}, $\theta(O^+(K))=Q(P(U))Q(P(W)){\Q_2^\times}^2$, where $P(U)$ (resp. $W$) is the set of primitive anisotropic vectors in U (resp. $W$) whose associated symmetries are in $O(U)$ (resp. $O(W)$). Then, $Q(P(U))=\theta(O^+(U)){\Q_2^\times}^2$ and $Q(P(W))=2^i\gamma\theta(O^+(\lan1,2^{j-i}\gamma\mu\ran)){\Q_2^\times}^2$ can be computed using \cite[1.9]{EH75}.  Since scaling does not affect the spinor norm, it will often be easier to compute the spinor norm of $M_2$ after scaling by $\epsilon$; to that end, we define $L:=M^\epsilon$ to be the lattice $M$ scaled by $\epsilon$, which simply means that $L$ and $M$ have the same basis, but $L$ is endowed with the quadratic map $Q^\epsilon(x)=\epsilon Q(x)$, where $Q$ is the quadratic map on $M$. 
 
\section{Proof of Main Theorem}\label{S4}

For the proof of the main theorem, it will be necessary to establish the following additional notation.  We define $\bar{Q}:=\frac{1}{2^\beta}Q$, and therefore $\bar{Q}(\nu)\in \Z_2^\times$.  Furthermore, for any $x\in N$, we have $\bar{B}(\nu,x)=\frac{1}{2^\beta}B(\nu,x)$ and therefore $\bar{B}(\nu,\bar{N})=\frac{1}{2^\beta}B(\nu,N)$, where $\bar{N}$ denotes the $\Z$-lattice with quadratic map $\bar{Q}$.  Thus, if we have $H(x)=\frac{Q(x)+2B(\nu,x)}{2^\alpha}$, where $\n(\nu,N)=2^\alpha\Z$ in the usual way, then 
\[
H(x)=\frac{2^\beta\bar{Q}(x)+2^{\beta+1}\bar{B}(\nu,x)}{2^\alpha}=\frac{\bar{Q}(x)+2\bar{B}(\nu,x)}{2^{\alpha-\beta}},
\]  
where $\alpha-\beta=1,2,3$. 

Since $\frac{1}{2^\beta}\in \Z_p$ for $p$ odd, therefore $N_p$ represents all of $\Z_p$ if and only if $\bar{N}_p$ represents all of $\Z_p$, when $p$ is odd. Since $dN=2^{3\beta}d\bar{N}$, it follows that $\ord_2(dN)-3\beta=\ord_2(d\bar{N})$.  Since $\rad(dN)'$ refers only to the odd part of $dN$, scaling by $2^\beta$ will not change this value, and we can use $\rad(dN)'$ and $\rad(d\bar{N})'$ interchangeably. We note that the integer  $\lambda$ defined in section 1 is equivalent to $\delta_{\bar{N}}$.   

Under this construction, we also have $\bar{M}:=\frac{1}{2^\beta}M$, which can also be obtained by the usual method, setting $\bar{M}=\Z\nu+\bar{N}$.  If we suppose that $N_p$ represents all $\Z_p$ whenever $p$ is odd, then $\bar{M}$ satisfies the hypotheses of Lemmas \ref{L0}, \ref{L6} and \ref{L7}.  Unless otherwise noted, we let $E$ denote $\Q(\sqrt{-td\overline{M}})$, where $t$ is some odd primitive spinor exception of $\gen(\bar{M})$.  

\begin{pot}
We will suppose throughout that $N_p$ represents all $\Z_p$ whenever $p$ is odd.  Suppose that $\alpha=\beta+1$.  First we will consider the case when $B(\nu,N_2)=2^{\beta-1}\Z_2$, and thus $N_2$ is diagonalizable by Remark \ref{RL1.1.2}.  Since $\n(\nu,N)=2^\alpha \Z$, we have $\s(N)=\n(N)=2^\beta\Z$.  Thus, $N_2\cong\lan 2^\beta\eta,2^i\gamma,2^j\mu\ran$ in a basis $\{e_1,e_2,e_3\}$, where $\beta\leq i\leq j$.  Letting $\nu=\frac{ae_1+be_2+ce_3}{2}$ with $0\leq a,b,c\leq 1$, it follows that for any choice of $a$ we must have $b=1$.   Therefore, $\{\nu,e_1,e_3\}$ is a basis for $M$.  Consider the sublattice $R:=\Z[\nu,2e_1,2e_3]$ of $M$.  Then, $R_p=M_p$ for every odd prime $p$, and therefore $\Z_p^\times\subseteq \theta(O^+(R_p))$ for every odd $p$, by Lemma \ref{L0}.  Furthermore, $\n(R_2\cap N_2)=\n(\Z_2[2\nu,2e_1,2e_3])=2^{\beta+2}\Z$, and therefore any representation of $2^\beta\epsilon+2^\alpha n$ by $R$ must come from the coset $\nu+N$.  

Thus we have 
\[
R_2^\frac{1}{2^{\beta}}\cong\begin{bmatrix}
\epsilon & \eta a & 2^{j-\beta}\mu c\\
 \eta a & 4\eta & 0\\
 2^{j-\beta}\mu c & 0 & 2^{j+2-\beta}\mu. 
\end{bmatrix}
\]
When $a=1$, then $R_2^\frac{1}{2^\beta}$ contains a sublattice isometric to  $\lan\epsilon, \epsilon(4\eta\epsilon-1)\ran$, which clearly represents all units in $\Z_2^\times$.  On the other hand, when $a=0$, then from the shape of $\nu$ we know that $j=\beta$ and $c=1$, and therefore $R_2$ again contains the sublattice $\lan\epsilon,\epsilon(4\mu\epsilon-1)\ran$, which again represents all units in $\Z_2^\times$.  Therefore, we get that $\Z_2^\times\in \theta(O^+(R_2))$, and hence $\Z_p^\times\subseteq \theta(O^+(R_p))$ for every prime $p$ implying that $\gen(R)$ has only one spinor genus \cite[102:9]{OM}.  Furthermore, $2^\beta\epsilon+2^\alpha n$ is represented primitively by $\gen(R)$ for every choice of $n$.  Therefore $H(x)$ is almost universal, by \cite[Corollary]{DS-P90}.

Next, we suppose that $B(\nu,N_2)\subseteq 2^\beta$, and hence $\n(N)\subseteq 2^{\beta+1}\Z$.  Therefore, any representation of $2^\beta\epsilon+2^\alpha n$ by $M$ is guaranteed to come from the coset $\nu+N$.  Furthermore, $\bar{Q}(\nu)\in \Z_2^\times$ and $2^\beta\n(\bar{N})=\n(N)$.  If part 1(b)(i) holds, then $2\s(N)=\n(N)=2^{\beta+2}\Z$ implies that $2\s(\bar{N})=\n(\bar{N})=4\Z$.  If part 1(b)(ii) holds, then $\bar{N}_2$ is diagonalizable, and $\ord_2(dN)=3(\beta+1)$ implies that $\ord_2(d\bar{N})=3$.  Similarly, if part 1(b)(iii) holds, then $\bar{N}_2$ is diagonalizable, $\ord_2(d\bar{N})=5$ and $B(\nu,\bar{N}_2)=2\Z$.  Therefore, when any of these parts hold, then $\gen(\bar{M})$ primitively represents all units in $\Z_2^\times$ by Lemma \ref{L6}.  Since $\Z_p^\times\subseteq \theta(O^+(\bar{M}_p))$ for every prime $p$, it follows that $\gen(\bar{M})$ only has one spinor genus.  Therefore, $\epsilon+2n$ is represented by $\bar{M}$ for $n$ sufficiently large by \cite[Corollary]{DS-P90}.  And hence, $M$ represents $2^\beta\epsilon+2^\alpha n$ for all $n$ sufficiently large, and therefore $H(x)$ is almost universal. 

Suppose that $\alpha=\beta+1$ and parts 1(a) and 1(b) both fail.  Then, it follows immediately from Lemma \ref{L6} that $H(x)$ is not almost universal. 

For part (2), we will suppose that $\alpha=\beta+2$.  Then $B(\nu,N_2)\subseteq 2^{\beta}\Z_2$ by Lemma \ref{L1.1}, and therefore $\n(N)\subseteq 2^{\beta+1}\Z$. Consequently, any representation of $2^\beta\epsilon+2^\alpha n$ by $M$ must be from the coset $\nu+N$.  

When $B(\nu,N_2)=2^\beta\Z_2$, then $\n(N)=2^{\beta+1}\Z$, and from Remark \ref{RL1.1.1} we know that $N_2$, and hence $\bar{N_2}$, has an orthogonal decomposition. Therefore, we have $\bar{N}_2\cong\lan2\eta,2^i\gamma,2^j\mu\ran$ in a basis $\{e_1,e_2,e_3\}$, where $1\leq i\leq j$.  As discussed in the proof of Lemma \ref{L7}, we may assume that $i=1$, and $\nu=\frac{ae_1+e_2+ce_3}{2}$ where $0\leq a,c\leq 1$.  If $a=0$, then $\bar{Q}(e_1)+2\bar{B}(\nu,e_1)\in 2\Z$, which cannot happen since $\n(\nu,\bar{N})=4\Z$.  Therefore, $a=1$, and $\bar{M}_2\cong\Z[\nu,e_2,e_3]$.  In this basis, 
\[
\bar{M}_2\cong\begin{bmatrix}
\epsilon & \gamma & 2^{j-1}\mu c\\
\gamma  & 2\gamma & 0\\
2^{j-1}\mu c & 0 & 2^j \mu,
\end{bmatrix}.
\]
If $c=0$, then $\epsilon=\frac{\eta+\gamma}{2}$ and therefore $\bar{M}_2\cong\lan\epsilon, \epsilon\eta\gamma, 2^j\mu\ran$ with $j\geq 1$.  If $c=1$, then $\epsilon=\frac{\eta+\gamma}{2}+2^{j-2}\mu$, and in the basis $\{\nu,-\gamma \nu+\epsilon e_2,2^{j-1}\mu e_1-\eta e_3\}$,
\[
\bar{M}_2\cong\lan \epsilon, \epsilon\gamma(\eta+2^{j-1}\mu),2^j\eta\mu(\eta+2^{j-1}\mu)\ran.
\]
In any case, $d\bar{M}=2^j\eta\gamma\mu$.

Suppose that part (i) of 2(a) holds, then $\ord_2(dN)-3\beta$ is odd, and hence $\ord_2(d\bar{N})$ is odd.  Suppose that $t$ is an odd primitive spinor exception of $\gen(\bar{M})$.  Then, $E=\Q(\sqrt{-2})$, and we are in the situation outlined in \cite[Theorem 2(c)]{EH94}. But $\bar{M}_2$ has a binary unimodular component, so $``r"$ there is $0$, and so $\theta(O^+(\bar{M}_2))\not\subseteq \mathfrak{N}_2(E)$.  Therefore, $\gen(\bar{M})$ has no odd primitive spinor exceptions.  In particular, if $\ord_2(d\bar{N})$ is odd, then $\epsilon+4 n$ is not a primitive spinor exception of $\gen(\bar{M})$.  

Now, suppose that (i) fails, meaning that $\ord_2(d\bar{N})$ is even and $E=\Q(\sqrt{-1})$.  Suppose that part (ii) of 2(a) holds, and hence $\ord_2(d\bar{N})=4$, and therefore $j=2$.  So either $\epsilon=\frac{\eta+\gamma}{2}$, in which case $1+\eta\gamma\equiv 2\mod 4$, or $\epsilon=\frac{\eta+\gamma}{2}+\mu$ and in this case $1+\eta\gamma\equiv 0\mod 4$. In either case, $\bar{L}_2:=\frac{1}{2^\beta}L_2$ contains a binary unimodular component of the form $\lan1,\xi\ran$, where $\xi$ is a unit with $1+\xi\equiv 2\mod 4$.  Since $j=2$, from \cite[1.2]{EH78}, we get $\Q_2^\times=\theta(O^+(\bar{M}_2))$.  Therefore, $\theta(O^+(\bar{M}_2))\not\subseteq \mathfrak{N}_2(E)$, and thus $\epsilon+4n$ is not a primitive spinor exception of $\gen(\bar{M})$.  

Suppose that parts (i) and (ii) of 2(a) both fail, so $\ord_2(d\bar{N})\geq 6$ is even, and now regardless of our choice for $c$, $\bar{M}_2\cong \lan\epsilon,\epsilon\eta\gamma,2^j\mu\ran$. If part (iii) of 2(a) holds, then by Lemma \ref{L9}, $\gen(\bar{M})$ has no odd primitive spinor exceptions.  

Suppose now parts (i)-(iii) of 2(a) fail.  Then $N_2\cong\lan2^{\beta+1}\eta,2^{\beta+1}\gamma,2^{\beta+j}\mu\ran$, where $j> 2$ is even. Suppose that part (iv) holds, that is, $\eta\gamma\equiv 5\mod 8$.  Then, since $j$ is even, the quadratic space underlying $\bar{L}_2$ is $[1,5,\epsilon\mu]$.  From the failure of (iii), $\mu\equiv 1\mod 4$, but then, $\epsilon\equiv 1\mod 4$, or else $\bar{M}_2$ is isotropic.  Therefore, computing the spinor norm of $\bar{M}_2$ using \cite[1.2]{EH78}, we get $\theta(O^+(\bar{M}_2))=\{1,5,6,14\}{\Q_2^\times}^2\not\subseteq \mathfrak{N}_2(E)$, which means that $\epsilon+4n$ is not a primitive spinor exception of $\gen(\bar{M})$ in this case.

Now suppose that we are in part 2(b), so $B(\nu,N_2)=2^{\beta+1}\Z_2$ and $\n(G_2)=2^{\beta+2}\Z_2$, where $G_2$ is the orthogonal complement of $\nu$ in $N_2$.  Since $\alpha=\beta+2$, this implies that $\n(N)=2^{\beta+2}\Z$, and therefore, $\Z[2\nu]\cong\lan2^{\beta+2}\epsilon\ran$ splits $N_2$ as an orthogonal summand, and hence $M_2\cong\lan2^\beta\epsilon\ran\perp G_2$, with $\n(G_2)=2^{\beta+2}\Z_2$.  We immediately rule out the possibilty that $G_2\cong2^{\beta+1}\A$, since this would mean that $M_2$ is isotropic, therefore $\bar{M}_2\cong\lan\epsilon,4\gamma,2^{j}\mu\ran$ in a basis $\{e_1,e_2,e_3\}$, where $j\geq 2$.   

Suppose that part (i) of 2(b) holds; that is, suppose that $\ord_2(d\bar{N})$ is odd.  Now using \cite[Theorem 2(b)]{EH94}, we have $\theta^*(\bar{M}_2,t)\neq \mathfrak{N}_2(E)$, and thus $\epsilon+4n$ cannot be a primitive spinor exception of $\gen(\bar{M})$.  

Suppose that part (ii) of 2(b) holds; that is, $\ord_2(d\bar{N})=6$, and hence  $\bar{M}_2\cong\lan\epsilon, 4\gamma,4\mu\ran$ with $\gamma\mu\equiv 1\mod 4$ since $\bar{M}_2$ is anisotropic.  Therefore, we may use \cite[1.2]{EH78} to show that $\theta(O^+(\bar{M}_2))=\Q_2^\times$.  Hence, in this case $\theta(O^+(\bar{M}_2))\not\subseteq \mathfrak{N}_2(E)$, and hence $\epsilon+4n$ is not a primitive spinor exception of $\gen(\bar{M})$.  

Suppose that part (iii) or 2(b) holds; that is, suppose that $\rad(d\bar{N})'$ is divisible by a prime $q\equiv 3\mod 4$.  Then it is immediate from Lemma \ref{L9} that $\gen(\bar{M})$ has no odd primitive spinor exceptions.

Now we will suppose that $\alpha=\beta+3$, and recall from Lemma \ref{L1.1} that this implies $B(\nu,N_2)=2^{\beta+1}\Z_2$.  Therefore $\n(N)=2^{\beta+2}\Z$, and as in the $\alpha=\beta+2$ case, we get $\bar{M}_2\cong\lan\epsilon\ran\perp \bar{G}_2$ in the basis $\{\nu,f_1,f_2\}$ where the $f_i$ are some appropriately chosen linear combination of the $e_i$.  Again, we are guaranteed that any representation of $2^\beta\epsilon+2^\alpha n$ by $M$ must come from the coset $\nu+N$.

Suppose that 3(a) holds, then $\bar{G}_2\cong 2^i\A$, where $i\geq 2$ is even since $\bar{M}_2$ is anisotropic.  But then for any $\rho\in \Z_2^\times$ there is a vector $\nu_\rho\in 2^i\A$ such that $Q(\nu_\rho)=2^{i+1}\rho$, which implies that $\Z_2^\times\subseteq \theta(O^+(\bar{M}_2))$.  Now,  $\Z_p^\times\subseteq \theta(O^+(\bar{M}_p))$ for every prime $p$, and therefore $\gen(\bar{M})$ has no primitive spinor exceptions. 

Suppose that part 3(a) fails, and 3(b) holds.  Then $\bar{G}_2$ is proper with $\n(\bar{G}_2)=8\Z_2$, and hence $\bar{M}_2\cong\lan \epsilon, 8b_1, 2^jb_2\ran,$ in the basis $\{\nu,f_1,f_2\}$ where the $f_i$ are some appropriately chosen linear combination of the $e_i$, and $b_1,b_2\in \Z_2^\times$ and $j\geq 3$.  If $\ord_2(d\bar{N})$ is even, then using \cite[Theorem 2(b)]{EH94} we get that $\theta^*(\bar{M}_2,t)\neq \mathfrak{N}_2(E)$.  In this case, $\gen(\bar{M})$ has no odd primitive spinor exceptions.  If $\ord_2(d\bar{N})=9$, then $j=4$ and hence $\bar{M}_2$ is of {\em Type E} as defined in \cite[page 531]{EH78}.  In this case, $\theta(O^+(\bar{M}_2))={\Q_2^\times}$, so $\gen(\bar{M})$ has no odd primitive spinor exceptions.  

Suppose that 3(a)-(b) fail and 3(c) holds.  Then $\ord_2(d\bar{N})$ is odd, and we are in the setting of \cite[Theorem 2(b)]{EH94} and since $``r"$ there equals 4, $\theta^*(\bar{M}_2,t)\neq \mathfrak{N}_2(E)$ for any odd $t$.  Hence, $\gen(\bar{M})$ has no odd primitive spinor exceptions in this case.  

If 3(a)-(c) all fail and 3(d) holds, then from Lemma \ref{L9}, $\gen(\bar{M})$ has no odd primitive spinor exceptions. 

Suppose that 3(a)-(d) all fail, and 3(e) holds.  Then, $\bar{M}_2\cong\lan\epsilon, 2^i\gamma, 2^j\mu\ran$ with $\epsilon\gamma\mu\not\equiv \epsilon\mod 8$.  Consequently, $(\epsilon+8n)\epsilon\gamma\mu\not\equiv 1\mod 8$ for any positive integer $n$. Therefore, $\Q(\sqrt{-(\epsilon+8n)d\bar{N}})\neq\Q(\sqrt{-\lambda})$, so in this case, $\epsilon+8n$ is not a primitive spinor exception of $\gen(\bar{M})$ for any positive integer $n$.  

Suppose that 3(a)-(e) all fail, and 3(f) holds.  Then, $\bar{M}_2\cong\lan\epsilon,8\gamma,2^j\gamma\ran$ with $j\geq 6$ even, and $\epsilon\not\equiv \gamma\mod 8$.  If $\epsilon\gamma\equiv -1,5\mod 8$ then $\bar{M}_2$ is anisotropic, therefore we may suppose that $\epsilon\gamma\equiv 3\mod 8$.  Now, using \cite[1.9]{EH75} to compute the spinor norm of $\bar{L}_2\cong\lan1,24,3\cdot2^j\ran$, we see that for any choice of $j$, $\theta (O^+(\bar{M}_2))$ contains $-6$.  Since $\ord_2(d\bar{N})$ here is odd, $-6\not\in \mathfrak{N}_2(E)$, so $\gen(\bar{M})$ has no odd primitive spinor exceptions in this case.  

Therefore, when any part of (2) or (3) holds, we have shown that $\epsilon+2^{\alpha-\beta} n$ is not a primitive spinor exception of $\gen(\bar{M})$ for any $n$.  Therefore $\epsilon+2^{\alpha-\beta}n$ is represented primitively by $\spn(\bar{M})$, and therefore by $\bar{M}$ itself for all $n$ sufficiently large, according to \cite[corollary]{DS-P90}.  Therefore $2^\beta\epsilon+2^\alpha n$ is represented by $M$ itself for all but finitely many $n$.  But any representation of $2^\beta\epsilon+2^\alpha n$ by $M$ must come from the coset $\nu+N$, and hence $\nu+N$ represents $2^\beta \epsilon+2^\alpha n$ for all but finitely many $n$.  Therefore, $H(x)$ is almost universal. 

Now we will show that when $\alpha=\beta+2$ or $\alpha=\beta+3$ but parts (2) and (3) fail, then $\rad(dN)'$ is a primitive spinor exception of $\gen(\bar{M})$.  It follows from Lemma \ref{L0} part (1) that $\rad(dN)'$ is primitively represented by $\bar{M}_p$ for every odd prime $p$.  When $\alpha=\beta+2$ and $B(\nu,N_2)=2^{\beta}\Z_2$, then from the failure of (2), $\bar{M}_2\cong\lan\epsilon,\epsilon\eta\gamma,2^j\mu\ran$ in a basis $\{e_1,e_2,e_3\}$.  Therefore $\bar{L}_2$, which must be anisotropic, has underlying quadratic space $[1,\eta\gamma,\epsilon\mu]$, and thus $\epsilon\rad(dN)'\equiv 1\mod 4$.  If $B(\nu,N_2)=2^{\beta+1}\Z_2$, then the underlying quadratic space of $\bar{M}_2$ is $[\epsilon,\gamma,\mu]$.  Since $M_2$ is anisotropic, therefore $\gamma\mu\equiv 1\mod 4$, and hence $\epsilon\equiv 1\mod 4$ by the failure of 2(b)(iii).  When $\alpha=\beta+3$, then $\epsilon\equiv\rad(dN)'\mod 8$ from the failure of 3(e). Hence in all cases we have $\epsilon\equiv \rad(dN)'\mod 2^{\alpha-\beta}$, and combining this with Lemma \ref{L7}, we get that $\rad(dN)'$ is represented primitively by $\bar{M}_2$.  Therefore, $\rad(dN)'$ is represented primitively by $\gen(\bar{M})$. 

In what follows, we will let $E$ denote $\Q\left(\sqrt{-\rad(dN)'d\overline{N}}\right)=\Q(\sqrt{-\lambda})$.  At the primes $p$ where $\big(\frac{-\lambda}{p}\big)=1$, we have $E_p=\Q_p$, and therefore $\mathfrak{N}_p(E)=\Q_p^\times$.  But now it follows immediately from the containments given in equation (3) of \cite{EH94} that $\theta(O^+(\bar{M}_p))=\mathfrak{N}_p(E)=\theta^*(\bar{M}_p,\rad(dN)')$.  At the primes $p$ for which $\big(\frac{-\lambda}{p}\big)=-1$, it follows easily from \cite[Theorem 1]{EH94} that $\theta(O^+(\bar{M}_p))\subseteq \mathfrak{N}_p(E),$ and $\mathfrak{N}_p(E)=\theta^*(\bar{M}_p,\rad(dN)')$. 

To compute the spinor norm and relative spinor norm of $\bar{M}_2$, it will be helpful to consider separately that cases for $\alpha=\beta+2$, $B(\nu,N_2)=2^\beta\Z_2$ or $2^{\beta+1}\Z_2$, and $\alpha=\beta+3$.  Let us first consider the case when $\alpha=\beta+2$ and $B(\nu,N_2)=2^\beta\Z_2$.  From our previous discussion, $\bar{L}_2\cong\lan 1,1,2^j\epsilon\mu\ran$, with $j\geq 4$ even, and $\epsilon\mu\equiv 1\mod 4$.   Using \cite[1.2]{EH78}, since our binary component is neither even nor odd, 
\[
\theta(O^+(\bar{M}_2))=\{\rho\in \Q_2^\times:(\rho, -1)=1\}=\{1,2,5,10\}{\Q_2^\times}^2=\mathfrak{N}_2(E).
\]
Now, using \cite[Theorem 2(b)]{EH94} we compute the relative spinor norm. The $``K"$ and $``K'"$ in that theorem are $K\cong\lan 2^{-2}\epsilon,\epsilon,2^j\gamma\ran$, and $K'=\bar{M}_2$.  Since $j$ is even, and clearly $\theta(O^+(K'))= \theta(O^+(\bar{M}_2))\subseteq \mathfrak{N}_2(E)$, parts (ii) and (iv) fail immediately.  Moreover, part (iii) fails, since $j\geq 4$.  Also, since $\rad(dN)'$ is odd, it is not contained in any $\Z_2$-ideal generated by $2^j$.  Using \cite[1.9]{EH75}, we get $\theta(O^+(K))\subseteq \{1,5\}{\Q_2^\times}^2\subseteq \mathfrak{N}_2(E)$, and therefore (i) of \cite[Theorem 2(b)]{EH94} fails.  Therefore, $\theta^*(\bar{M}_2,\rad(dN)')=\mathfrak{N}_2(E)$ when $B(\nu,N_2)=\Z_2$.  

Now we deal with the case where $\alpha=\beta+2$ and $B(\nu,N_2)=2^{\beta+1}\Z_2$.  From previous discussion, $\bar{M}_2\cong\lan\epsilon,4\gamma,2^j\mu\ran$, with $j\geq 4$ even, and $\epsilon\gamma\mu\equiv 1\mod 4$.  Note that $\bar{M}_2$ is anisotropic, so immediately $\epsilon\gamma\equiv 1\mod 4$.  Now using \cite[1.9]{EH75}, we get $\theta(O^+(\bar{M}_2))\subseteq \mathfrak{N}_2(E)$.  Using \cite[Theorem 2(b)]{EH94} we have $``K"$ and $``K'"$ in that theorem are $K=\bar{M}_2$ and $K'\cong\lan4\epsilon,4\beta,2^j\gamma\ran$.  We have $``r"=2$ here, which is even, and $\theta(O^+(K))\subseteq \mathfrak{N}_2(E)$, so parts (i) and (iv) fail.  Furthermore, $\rad(dN)'$ is odd, and is therefore not contained in any $\Z_2$-ideal generated by $2^{2}$ or $2^{j}$.  So, we may conclude that $\theta^*(\bar{M}_2,\rad(dN)')=\mathfrak{N}_2(E)$.  

Next we will deal with the case when $\alpha=\beta+3$, and hence $\bar{M}_2\cong\lan\epsilon,2^i\gamma,2^j\gamma\ran$ with $i\leq j$ by the failure of 3(a) and 3(e). First we will consider the case where $i=j$, so we may assume that $i\geq 4$, from the failure of 3(b).  Thus, using \cite[1.2]{EH78}, 
\[
\theta(O^+(\bar{M}_2))=\{\rho\in \Q_2^\times:(\rho,-1)=1\}=\{1,2,5,10\}{\Q_2^\times}^2,
\]
which is equal to $\mathfrak{N}_2(E)$. Furthermore, $\theta^*(\bar{M}_2,\rad(dN)')=\mathfrak{N}_2(E)$, since $i\geq 4$ and hence parts (i)-(iv) of \cite[Theorem 2(b)]{EH94} immediately fail.

If $i\neq j$, then $\bar{L}_2\cong\lan1,2^i\epsilon\gamma,2^j\epsilon\gamma\ran$, where $\epsilon\gamma\equiv 1\mod 4$ when $ord_2(d\bar{N})$ is even, and $\epsilon\gamma\equiv 1,3\mod 8$ when $\ord_2(d\bar{N})$ is odd. When $i$ and $j$ have the same parity, then computing the spinor norm using \cite[1.9]{EH75}, with $U\cong\lan 1,2^i\epsilon\gamma\ran$ and $W\cong2^i\epsilon\gamma\lan1,2^{j-i}\ran$, we get
\[
Q(P(U))\subseteq \{1,5,2^i\epsilon\gamma\}{\Q_2^\times}^2 \text{ and }Q(P(W))\subseteq2^i\epsilon\gamma \{1,5\}{\Q_2^\times}^2.
\]
Therefore, $\theta(O^+(\bar{M}_2))\subseteq\mathfrak{N}_2(E)$ regardless of our choice for $\epsilon\gamma$. When $i$ and $j$ have opposite parity, then using the same $U$ and $W$ as above, we get 
\[
Q(P(U))\subseteq \{1,2,3,6\}{\Q_2^\times}^2 \text{ and }Q(P(W))\subseteq2^{i}\epsilon\gamma \{1,2,3\}{\Q_2^\times}^2,
\]
and again $\theta(O^+(\bar{M}_2))\subseteq\mathfrak{N}_2(E)$ for any choice of $\epsilon\gamma\equiv 1,3\mod 8$.  Finally, given the failure of part (3), it can be easily verified that $\theta^*(\bar{M}_2,\rad(dN)')\subseteq \mathfrak{N}_2(E)$ using \cite[Theorem 2]{EH94}.

We have shown that when parts (2) and (3) fail, then $\rad(dN)'$ is a primitive spinor exception of $\gen(\bar{M})$.  Now we will show that when (1), (2), and (3) all fail, $H(x)$ is almost universal if and only if (4) holds. 

Suppose that (1), (2), and (3) all fail, and (4) holds.  Then, $\frac{2^\beta\rad(dN)'-2^\beta\epsilon}{2^\alpha}$ is represented by $H(x)$ for $\alpha=\beta+2,\beta+3$; hence $\rad(dN)'$ is represented by the coset $\nu+\bar{N}$, and therefore by the lattice $\bar{M}$.  If $\epsilon+2^{\alpha-\beta} n$ is not a primitive spinor exception of $\gen(\bar{M})$, then it is represented by $\bar{M}$ for $n$ sufficiently large, by \cite[Corollary]{DS-P90}.  Consequently, $2^\beta\epsilon+2^\alpha n$ is represented by $M$ for all $n$ sufficiently large.  If $\epsilon+2^\alpha n$ is a primitive spinor exception of $\gen(\bar{M})$, then $\epsilon+2^\alpha n=m^2\rad(dN)'$ for some integer $m$, and therefore $\epsilon+2^\alpha n$ is represented by the lattice $\bar{M}$ for all $n$, implying that $2^\beta\epsilon+2^\alpha n$ is represented by $M$ for all $n$.  Therefore, when (4) holds, then $H(x)$ is almost universal.  

Now, suppose that $\alpha=\beta+2,\beta+3$ and (2)-(4) all fail.  Then, $\frac{2^\beta\rad(dN)'-2^\beta\epsilon}{2^\alpha}$ is not represented by $H(x)$, and therefore $\rad(dN)'$ is not represented by $\bar{M}$.  Since we are assuming that (2)-(4) all fail, we know that $\rad(dN)'$ is a primitive spinor exception of $\gen(\bar{M})$.  Furthermore, from previous discussions, we know that when (2) and (3) fail, then  $\rad(dN)'\equiv \epsilon\mod 2^{\alpha-\beta}$.  Now, as shown in \cite{S-P00}, there exist infinitely many primes $q$ for which
\[
n:=\frac{\rad(dN)'q^2-\epsilon}{2^{\alpha-\beta}}
\]
is an integer not represented by $H(x)$.  Therefore, in this case $H(x)$ is not almost universal.  
\end{pot}

\section*{Acknowledgements}This work has been partially carried out at the Max Planck Institute for Mathematics in Bonn, Germany.  I wish to thank the MPIM for its support and hospitality during my stay.  In addition, I wish to extend my sincerest thanks to Wai Kiu Chan for his patience in advising the thesis in which this work originated. Thanks also to the referee for providing many helpful comments and suggestions.

\bibliographystyle{amsplain}
\bibliography{Haensch.bbl}

\end{document}